\documentclass[11pt]{amsart}
\usepackage{amssymb,amsmath,amsfonts,amsthm}
\usepackage{graphicx}
\usepackage{color}
\usepackage{psfrag}
\usepackage{hyperref}
\usepackage{soul}

\newtheorem{theorem}{Theorem}[section]
\newtheorem{lemma}[theorem]{Lemma}
\newtheorem{proposition}[theorem]{Proposition}

\theoremstyle{definition}
\newtheorem{definition}[theorem]{Definition}

\newtheorem{remark}[theorem]{Remark}

\newtheorem*{theoremb}{Theorem B}
\newtheorem*{theorema}{Theorem A}
\newtheorem{claim}{Claim}

\def\tor{\mathbb{T}}

\def\skr{\mathrm{Sk}^r(\mathbb{T}^2 \times \mathbb{T}^2)}
\def\sk{\mathrm{Sk}^3(\mathbb{T}^2 \times \mathbb{T}^2)}

\def\cP{\mathcal{P}}

\title[Uniqueness of $u$-Gibbs measures]{Uniqueness of $u$-Gibbs measures for  hyperbolic skew products on $\tor^4$}
\author{Sylvain Crovisier}
\author{Davi Obata}
\author{Mauricio Poletti}
\thanks{{S.C., D.O. and M.P. were partially supported by the ERC project 692925 NUHGD, M.P. was  supported by Instituto Serrapilheira, grant ``Jangada Din\^amica: Impulsionando Sistemas Din\^amicos na Regi\~ao Nordeste". and the Coordena\c{c}\~ao de Aperfei\c{c}oamento de Pessoal de N\'ivel Superior - Brasil (CAPES) - Finance Code 001.}}

\newcommand{\information}{{
  \bigskip
  \footnotesize
\textbf{Sylvain Crovisier}: \textsc{Laboratoire de Math\'ematiques d'Orsay, CNRS - UMR 8628, Universit\'e Paris-Saclay, 91405 Orsay, France.}\par\nopagebreak
\textit{E-mail:} \texttt{sylvain.crovisier@universite-paris-saclay.fr}
 \medskip
 
	\textbf{Davi Obata}: \textsc{Department of Mathematics,  275 TMCB Brigham Young University,  Provo, Utah, 84602.} \par\nopagebreak
  
  \textit{E-mail:} \texttt{davi.obata@mathematics.byu.edu}

  \medskip
  \textbf{Mauricio Poletti}: \textsc{Departamento de Matem\'atica, Universidade Federal do Cear\'a, Campus do PICI, Bloco 914, CEP 60455-760.  Fortaleza – CE, Brasil.}\par\nopagebreak
  \textit{E-mail:} \texttt{mpoletti@mat.ufc.br}
}}

\begin{document}

\begin{abstract}
We study the $u$-Gibbs measures of a certain class of uniformly hyperbolic skew products on $\mathbb{T}^4$. These systems have a strong unstable and a weak unstable direction.  Among such skew products, we show the existence of a subset which is $C^r$-dense and $C^2$-open for which every $u$-Gibbs measure is SRB. In particular, there is only one such measure.  As an application, we obtain the minimality of the strong unstable foliation. 
\end{abstract}

\maketitle

\section{Introduction}
%


Anosov (or uniformly hyperbolic) diffeomorphisms have been extensively studied and today we have a good description of its topological and statistical properties.  These properties are closely related to the properties of its invariant foliations. 

Consider a closed manifold $M$ supporting a transitive Anosov diffeomorphism  $f:M \to M$. Suppose that $f$ admits a splitting of the form $TM = E^s \oplus E^{wu} \oplus E^{uu}$, where $E^{wu}$ is uniformly expanding under the action of $Df$, but at a slower rate than $E^{uu}$.   In this case, we can consider the unstable foliation $\mathcal{F}^u$ which is tangent to $E^{wu} \oplus E^{uu}$, and the strong unstable foliation $\mathcal{F}^{uu}$ which is tangent to $E^{uu}$. 

The foliation $\mathcal{F}^u$ is well understood. Indeed, it is known that $\mathcal{F}^u$ is minimal (i.e. every leaf is dense).  Regarding ergodic properties of this foliation, there is only one invariant measure which admits conditional measures along $\mathcal{F}^u$-leaves that are absolutely continuous with respect to the Lebesgue measure on these leaves \cite{sinai, bowenbook, ruelle}.  This measure is called SRB (Sinai-Ruelle-Bowen),  (see Section \ref{subsec.srbugibbs} for the  precise definition).  The SRB measures were constructed first using ideas from statistical mechanics and they are expected to be the ones truly ``physically observed'' in a system (see \cite{young} for a discussion on this).  Questions on existence and uniqueness of these measures are on the core to understanding the statistical behavior of dynamical systems, see \cite{Pa00}. 


On the other hand,  the foliation $\mathcal{F}^{uu}$ is not well understood.  For instance, it was only recently announced in a joint work of the first author with Avila, Eskin, Potrie, Wilkinson and Zhang,  that $\mathcal{F}^{uu}$ is minimal for any Anosov $C^{1+ \alpha}$-diffeomorphism of $\tor^3$.  In higher dimensions,  the first author jointly with Avila and Wilkinson,  recently announced that $C^1$-open and $C^r$-dense among the transitive Anosov $C^r$-diffeomorphisms admitting a decomposition $E^s \oplus E^c \oplus E^u$, where $E^c$ is one dimensional and uniformly expanding,  the strong unstable foliation is minimal.

 A class of invariant measures associated to $\mathcal{F}^{uu}$ are the  $u$-Gibbs measures.  A $u$-Gibbs measure is a measure that admits conditional measures along $\mathcal{F}^{uu}$-leaves that are absolutely continuous with respect to the Lebesgue measure on these leaves (see Section \ref{subsec.srbugibbs}).  Let us make a few remarks:
 \begin{itemize}
 \item contrary to the SRB measures, $u$-Gibbs measures do not have to be unique;
 \item SRB measures are $u$-Gibbs;
 \item $u$-Gibbs measures do not have to be SRB.
 \end{itemize}

An interesting problem is to understand conditions that guarantee that a $u$-Gibbs measure is SRB.  There have been some recent progress for this problem. In \cite{alos}, it is proved that for an Anosov  $C^2$-diffeomorphism near a volume preserving one in $\tor^3$, either the directions $E^s \oplus E^{uu}$ are tangent to a two-dimensional foliation, or every fully supported $u$-Gibbs measure is SRB.  In \cite{katz}, the author gives a condition called quantified non integrability (QNI) that guarantees that a $u$-Gibbs measure is SRB whenever the direction $E^{wu}$ is one dimensional.   In \cite{EskinPotrieZha},  the authors obtained equivalent notions to QNI that are easier to work with. Their result will be used by the first author with Avila et al. to prove that for any Anosov $C^{\infty}$-diffeomorphism in $\tor^3$ admitting a decomposition $E^{s} \oplus E^{wu} \oplus E^{uu}$  either $E^{s}$ and $E^{uu}$ are jointly integrable, or every $u$-Gibbs measure is SRB.   More related to the present work, in \cite{obata}, it is proved a type of classification of $u$-Gibbs measures for some partially hyperbolic skew products on $\tor^4$ (see Theorem \ref{thm.urigidityanosov} below). These results are inspired by some important  measure rigidity techniques from \cite{eskinm, el, benoistquint, brownhertz} to classify stationary measures of certain random products (see also \cite{cantatdujardin} for an interesting application of \cite{brownhertz}).

 In this work, we will study certain types of Anosov rigid skew products on $\tor^4$. Let us establish  our setting.  Let $\skr$  be the space of $C^r$-diffeomorphisms of the form
$$
\begin{array}{rcl}
f : \mathbb{T}^2 \times \mathbb{T}^2  & \to & \mathbb{T}^2 \times \mathbb{T}^2\\
(p_1, p_2) & \mapsto & (f_1(p_1), f_2(p_1,p_2)),
\end{array}
$$
where $f_1$ is a $C^r$-diffeomorphism of $\tor^2$ and for each $p_1\in \tor^2$, $f_2(p_1,.)$ is a $C^r$-diffeomorphism of $\tor^2$.  The term rigid skew product here refers to the fact that we ask that these systems preserve the smooth fibered structure of $\tor^2 \times \tor^2$. 

Let $\operatorname{Ph}^r \subset \skr$  be the set of partially hyperbolic $C^r$-diffeomorphisms $f$ such that
\begin{enumerate}
\item[(a)] $f$ admits a dominated decomposition $T\tor^4 = E^{ss} \oplus E^{ws} \oplus E^{wu} \oplus E^{uu}$, with $E^{ws} \oplus E^{wu}$ tangent to the vertical fibers $\{x\} \times \tor^2$ and there are constants
\begin{equation}\label{eq.hypconstants}
\chi^{ss}_- <\chi^{ss}_+ < \chi^{ws}_-  < 1 <  \chi^{wu}_+ < \chi^{uu}_- < \chi^{uu}_+,
\end{equation}
such that for every $p\in \tor^4$,
$$
\begin{array}{c}
\chi^{ss}_-\leq \|Df(p)|_{E^{ss}}\| \leq \chi^{ss}_+\\
\chi^{ws}_- \leq  \|Df(p)|_{E^{ws}}\| \leq  \|Df(p)|_{E^{wu}}\| \leq \chi^{wu}_+\\
\chi^{uu}_- \leq  \|Df(p) |_{E^{uu}}\|\leq \chi^{uu}_+.
\end{array}
$$
Write $E^c = E^{ws} \oplus E^{wu}$.

\item[(b)] $f$ is $2$-center bunched, that is, 
\[
\chi^{ss}_+ < \displaystyle \left( \frac{\chi^{ws}_-}{\chi^{wu}_+} \right)^2 \textrm{ and } \left( \frac{\chi^{wu}_+}{\chi^{ws}_-}\right)^2 < \chi^{uu}_-.
\]

\item[(c)] $f$ verifies
\[
\displaystyle \frac{\log \chi^{ws}_-}{\log \chi^{ss}_+} < \frac{\log \chi^{uu}_- - \log \chi^{wu}_+}{ - \log \chi^{ss}_-}.
\]
\end{enumerate}
Moreover, we define the set $\mathcal{A}^r $ as the set of Anosov diffeomorphisms in $\operatorname{Ph}^r$ which contracts $E^{ws}$ and expands $E^{wu}$ uniformly. Observe that $\mathcal{A}^r$ is  $C^1$-open in $\skr$. 

As we mentioned before, $u$-Gibbs measures do not have to be unique. Let us give an example in our setting.   Let 
\[
A = 
\begin{pmatrix}
2&1\\
1 & 1
\end{pmatrix},
\]
and consider
\[
B=
\begin{pmatrix}
A^5 & 0\\
0 & A
\end{pmatrix}
=
\begin{pmatrix}
89 & 55 & 0 & 0\\
55 & 34 & 0 & 0\\
0&0 & 2 & 1\\
0 & 0 & 1 & 1
\end{pmatrix}.
\]
Observe that $B \in \mathrm{SL}(4 ,\mathbb{Z})$ and $B$ induces a diffeomorphism  $f:\mathbb{T}^4 \to \mathbb{T}^4$ of the form
\[
\begin{array}{rcl}
f: \mathbb{T}^2 \times \mathbb{T}^2 & \to & \mathbb{T}^2 \times \mathbb{T}^2\\
(p_1, p_2)  & \mapsto & (f_1(p_1), f_2(p_2)),
\end{array}
\] 
where $f_1$ and $f_2$ are the diffeomorphisms induced by $A^5$ and $A$ on $\tor^2$, respectively.  It is easy to verify that $f\in \mathcal{A}^r$. 

Let $\mu_1$ be the Lebesgue measure on $\tor^2$, which coincides with the unique SRB measure of $f_1$, and let $\mu_2$ be any ergodic $f_2$-invariant measure.  The measure $\mu = \mu_1 \times \mu_2$ is a $u$-Gibbs measure. Therefore,  there are infinitely many different $u$-Gibbs measures for $f$.  However, our Theorem A below states that, in our setting, having infinitely many $u$-Gibbs measures is not generic.

\begin{theorema}
\label{thm.maintheorem}
For any $r\geq 3$, there exists a set $\mathcal{U}$, which is $C^r$-dense and $C^2$-open in $\mathcal{A}^r$, with the following property: if $f\in \mathcal{U}$ then $f$ has only one $u$-Gibbs measure. Moreover, this measure coincides with the unique  SRB measure of $f$. 
\end{theorema}

The proof of our Theorem A is based on the rigidity result stated in Theorem \ref{thm.urigidityanosov} below (see \cite{obata}). This theorem states that a $u$-Gibbs measure is either SRB, or there is a $2$-torus tangent to $E^{ss}$ and $E^{uu}$, or there is a type of ``infinitesimal'' rigidity for the system.  It is known that, in our setting, $C^r$-generically there is no $2$-torus tangent to the strong directions. 

The main goal of this paper is to develop $C^r$-perturbation techniques that break the ``infinitesimal'' rigidity of Theorem \ref{thm.urigidityanosov}. Actually our techniques show that a $u$-Gibbs measure with one positive and one negative center exponents are SRB in some open set of partially hyperbolic (not necessarily Anosov) rigid skew products, see Theorem~\ref{thm.technical}.

Using our Theorem A, we can also obtain the minimality of the strong unstable foliation. 

\begin{theoremb}\label{thm.minimality}
For any $r\geq 3$,  let $\mathcal{U}$ be the $C^r$-dense and $C^2$-open subset of $\mathcal{A}^r$ obtained in Theorem A.  Then, for any $f\in \mathcal{U}$, the strong unstable foliation is minimal. 
\end{theoremb}

We remark that our argument to prove the minimality of the strong unstable foliation in Theorem B is different from the arguments used in the works mentioned before (see Section \ref{section.minimality}).  

In the ongoing work for Anosov diffeomorphisms in $\tor^3$ of the first author with Avila et al, mentioned above, the authors first prove the minimality of the strong unstable foliation. Then, they use it to show the uniqueness of the $u$-Gibbs measure whenever $E^{s}$ and $E^{uu}$ are not jointly integrable. Our approach is in the converse direction.  We first show the uniqueness of the $u$-Gibbs measure, generically, and then we use it to show the minimality.  

\begin{remark}\label{rem.c1}
In both results the set is actually a $C^1$ open set inside subsets of $C^r$ maps with uniformly bounded $C^2$ norm. The bound on the $C^2$ norm is necessary to control the variation of the holonomies with respect to $f$ in the $C^1$ topology, see Section~\ref{subsec.holonomies}.
\end{remark}

The rigid skew product setting and conditions (b) and (c) are only used to apply Theorem \ref{thm.urigidityanosov}. Let us make a few remarks about it.  The rigid skew product setting is used to obtain that any $u$-Gibbs measure for the system projects to the unique SRB measure of the basis. In particular, the projected measure has a property called local product structure. This is a crucial property in the proof of Theorem \ref{thm.urigidityanosov} to build the $2$-torus tangent to the strong directions.
Condition (b) is related to higher regularity of the strong unstable and strong stable foliations when restricted to a center unstable manifold. Condition (c) implies that we can find a $\theta \in (0,1)$ such that the direction $E^{uu}$ is $\theta$-H\"older and $(\chi^{ss}_+)^{\theta} < \chi^{ws}_-$ (see the introduction of \cite{obata} for more details).

This work is organized as follows. In Section \ref{section.preliminaries}, we review some preliminary material that will be needed in our proofs. In Section \ref{section.proofa}, we prove Theorem A assuming Proposition \ref{prop.mainprop}, which is the main perturbative result of this paper.  The proof of Proposition \ref{prop.mainprop} is done throughout Sections \ref{section.transversality} and \ref{section.proofprop}.  Theorem B is then proved in Section \ref{section.minimality}.


%
%

\section{Preliminaries} \label{section.preliminaries}

Throughout this section $f$ is a $C^r$-rigid skew product that belongs to $\mathcal{A}^r$, for $r\geq 3$.  Recall that $f$ admits a $Df$-invariant decomposition of the form $T\mathbb{T}^4 = E^{ss} \oplus E^{ws} \oplus E^{wu} \oplus E^{uu}$.  In the rigid skew product setting, the direction $E^c = E^{ws} \oplus E^{wu}$ is tangent to the vertical tori of $\mathbb{T}^2 \times \mathbb{T}^2$.  In particular, the vertical tori give the center foliation $\mathcal{F}^c$.

  Let $E^s = E^{ss} \oplus E^{ws} $ and $E^{u} = E^{wu} \oplus E^{uu}$ be the stable and unstable directions, respectively. It is well known that the directions $E^*$ integrate in foliations $\mathcal{F}^*$, for $* = s, ss, u, uu$. We can also consider the weak stable and weak unstable foliations given by $\mathcal{F}^{ws} = \mathcal{F}^s \cap \mathcal{F}^c$ and $\mathcal{F}^{wu} = \mathcal{F}^u \cap \mathcal{F}^c$, respectively. 

For each foliation $\mathcal{F}^*$ and for any point $p\in \mathbb{T}^4$, we denote the leaf containing $p$ by $W^*(p,f)$. We will omit $f$ whenever it is clear which diffeomorphism we are referring to. 

For a point $x\in \tor^2$, we write $W^s(x,f_1)$ to be the stable manifold of $x$ for the Anosov diffeomorphism $f_1$. Similarly, we write $W^u(x,f_1)$ to be the unstable manifold of $x$ for $f_1$.

\subsection{SRB and $u$-Gibbs measures}\label{subsec.srbugibbs}

Let $\mu$ be an $f$-invariant measure. A partition $\xi$ of $\mathbb{T}^4$ is \emph{$\mu$-measurable}, if up to a set of $\mu$-measure zero, the quotient $\mathbb{T}^4/\xi$ is separated by a countable number of measurable sets. Denote by $\hat{\mu}$ the quotient measure on $\mathbb{T}^4/\xi$.

By Rokhlin's disintegration theorem \cite{rokhlin}, for a measurable partition $\xi$, there exists a set of conditional measures $\{\mu_D^{\xi}: D\in \xi\}$ such that for $\hat{\mu}$-almost every $D\in \xi$ the measure $\mu_D^{\xi}$ is a probability measure supported on $D$, for each measurable set $B\subset M$ the application $D \mapsto \mu^{\xi}_D(B)$ is measurable and
\begin{equation}
\label{ob.disintegration}
\mu(B) = \displaystyle \int_{M/\xi} \mu_D^{\xi}(B) d\hat{\mu}(D).
\end{equation}

We remark that usually the unstable partition $\{W^u(p)\}_{p\in \mathbb{T}^4}$ is not a measurable partition. We say that a $\mu$-measurable partition $\xi^u$ is \emph{$\mathcal{F}^u$-subordinated} if for for $\mu$-almost every $p$, the following conditions are verified:
\begin{itemize}
\item $\xi^u(p) \subset W^u(p)$;
\item $\xi^u(p)$ contains an open neighborhood of $p$ inside $W^u(p)$. 
\end{itemize}

\begin{definition}[SRB measure]
\label{defi.srb}
An $f$-invariant probability measure $\mu$ is \emph{SRB} if for any $\mathcal{F}^u$-subordinated measurable partition $\xi^u$, for $\mu$-almost every $p$, the conditional measure $\mu^{u}_{\xi^u(p)}$ is absolutely continuous with respect to the riemannian volume of $W^u(p)$.
\end{definition}

Analogously, we say that a $\mu$-measurable partition $\xi^{uu}$ is $\mathcal{F}^{uu}$-subordinated,  if for $\mu$-almost every $p$, $\xi^{uu}(p) \subset W^{uu}(p)$ and $\xi^{uu}(p)$ contains an open neighborhood of $p$ inside $W^{uu}(p)$. 

\begin{definition}[$u$-Gibbs measure]
\label{defi.ugibbs}
An $f$-invariant probability measure $\mu$ is \textbf{$u$-Gibbs} if for any $\mu$-measurable partition $\xi^{uu}$ which is $\mathcal{F}^{uu}$-subordinated, for $\mu$-almost every point $p$, the conditional measure $\mu^{uu}_{\xi^{uu}(p)}$ is absolutely continuous with respect to the riemannian volume  of $W^{uu}(p)$. 
\end{definition}

By the absolute continuity of the strong unstable foliation, we know that every SRB measure is $u$-Gibbs (see \cite{bdv}). However, not every $u$-Gibbs measure is SRB. 

\subsection{Holonomies and rigidity of $u$-Gibbs measures}\label{subsec.holonomies}

For $f\in \mathcal{A}^r$, consider the constants given by \eqref{eq.hypconstants} and observe that condition (b) from the definition of $\mathcal{A}^r$ implies that
\[
\displaystyle \chi^{ss}_+ < \frac{\chi^{ws}_-}{\chi^{wu}_+} \textrm{ and } \frac{\chi^{wu}_+}{\chi^{ws}_-} < \chi^{ss}_-.
\]
This condition is called \emph{center bunching}.

Given two points $p$ and $q$ in the same strong unstable manifold, we can define the unstable holonomy $H^u_{p,q,f}: W^c(p) \to W^c(q)$ as the unique point $H^u_{p,q,f}(x)\in W^{uu}(x,f) \cap W^c(q,f)$. In the rigid skew product setting, this map is well defined in the entire center manifold $W^c(p)$.  Similarly, define the stable holonomy $H^s_{p,z,f}:W^c(p) \to W^c(z)$, for any $p$ and $z$ in the same strong stable manifold. 

The center bunching condition implies that $H^u_{p,q,f}$ is $C^1$ (see \cite{psw1, psw2}). In particular, we can consider the linear unstable holonomy $DH^u_{p,q,f}: TW^c(p) \to TW^c(q)$.

Let us explain now why the holonomies $H^u_{p,q,f}$ varies continuously in the $C^1$-topology with the choices of $p$, $q$ and $f$.  For any $n\in \mathbb{N}$, consider the family
\[
\{f^n|_{W^c(f^{-n}(q))} \circ f^{-n}|_{W^c(p)}\}_{p\in \mathbb{T}^4, q\in W^{uu}_1(p,f)}.
\]
This family converges exponentially fast in the $C^1$-topology as $n$ increases and the limit gives the unstable holonomy $\{H^u_{p,q,f}\}_{p\in \mathbb{T}^4, q\in W^{uu}_1(p,f)}$.  The speed of convergence only depends on the constants given in \eqref{eq.hypconstants} and $\|f\|_{C^2}$ (see \cite{brown,  obata}). In particular, if $g$ is $C^2$-close to $f$, $p'$ is close to $p$ and $q'$ is close to $q$, then $H^u_{p',q',g}$ is $C^1$-close to $H^u_{p,q,f}$. More generally if the $C^2$ norm is uniformly bounded we only need $g$ to be $C^1$-close to $f$, recall Remark~\ref{rem.c1}.  

When there is no confusion on which map $f$ we are dealing with we omit the sub-index $f$.

The main ingredient in the proof of our Theorem A is the following result:
\begin{theorem}[\cite{obata}]
\label{thm.urigidityanosov}
For $r\geq 3$, let $f\in \mathcal{A}^r$  and let $\mu$ be an ergodic $u$-Gibbs measure. Then one of the following holds:
\begin{enumerate}
\item $\mu$ is an SRB measure;
\item There exists a set $X\subset \mathbb{T}^4$ with $\mu(X) =1$ such that for $p\in X$ and for any $q\in W^{uu}(p)$ 
\[
DH^u_{p,q}(p) E^{ws}(p) = E^{ws}(q);
\]
\item there are finitely many $2$-dimensional tori $T^{su}_1, \cdots, T^{su}_k$ tangent to $E^{ss} \oplus E^{uu}$ such that
\[
\mathrm{supp}(\mu) = \displaystyle \bigsqcup_{j=1}^k T^{su}_j.
\]
\end{enumerate}

\end{theorem}

\section{Proof of  the Theorem A}\label{section.proofa}

Let $f\in \mathcal{A}^r$ and fix points $q_u, p_1, p_2, p_3 \in \tor^2$ which are periodic for $f_1$. Fix points $q_i \in W^{u}(q_u,f_1) \cap W^s(p_i,f_1)$, for $i=1, \cdots, 3$.  For each $i,j \in \{u,1,2,3\}$, with $i \neq j$, let $H^u_{i,j,f}: W^c(q_i) \to W^c(q_j)$  be the unstable holonomy map between the center leaves of $q_i$ and $q_j$. See figure~\ref{figure-config}. 

\begin{figure}[h]
\centering
\includegraphics[width= 0.4\textwidth]{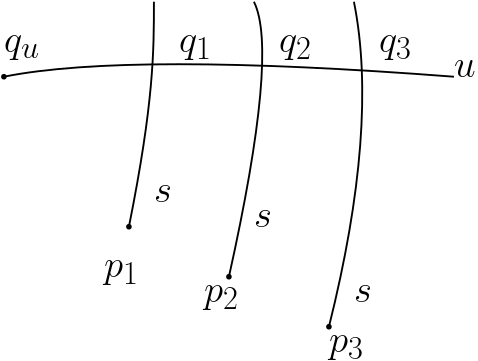}
\caption{Definition of $q_1,q_2,q_3$}
\label{figure-config}
\end{figure}

The main technical tool in our proof is given by the following proposition.
\begin{proposition}\label{prop.mainprop}
Let $f\in \mathcal{A}^r$, for $r\geq 3$, and fix $q_u, p_1, p_2,p_3, q_1, q_2, q_3\in \tor^2$ as above. Then for any $\varepsilon>0$, there exists $g\in \mathcal{A}^r$, with $d_{C^r}(f,g)<\varepsilon $, that verifies the following properties:
\begin{enumerate}

\item there are neighborhoods $V_l$ of $q_l$ in $\tor^2$, for $l=u, 1,2,3$,  such that  if $ x_u, x_1, x_2, x_3\in \tor^2$ are points verifying:
\begin{itemize}
\item each point $x_l$ belongs to $V_l$, for $l=u,1,2,3$, and
\item the points $x_u,x_1,x_2,x_3$ belong to the same unstable manifold for $g_1$,
\end{itemize}
then for any point $x\in \tor^2$, there exist $i,j\in \{u,1,2,3\}$ such that 
\[
DH^u_{x_j,x_l,g}(x)E^{ws}_g(x) \neq E^{ws}_g(H^u_{x_j,x_l,g}(x)).
\]
\item For any rigid skew product $h$ in a $C^2$-neighborhood of $g$, there are points $\tilde{p}_u, \tilde{q}_1, \tilde{q}_2, \tilde{q}_3\in \tor^2$ in the same unstable manifold for $h_1$ that verify property (1) above.
\end{enumerate}

\end{proposition}

\begin{figure}[h]
\centering
\includegraphics[width= 0.6\textwidth]{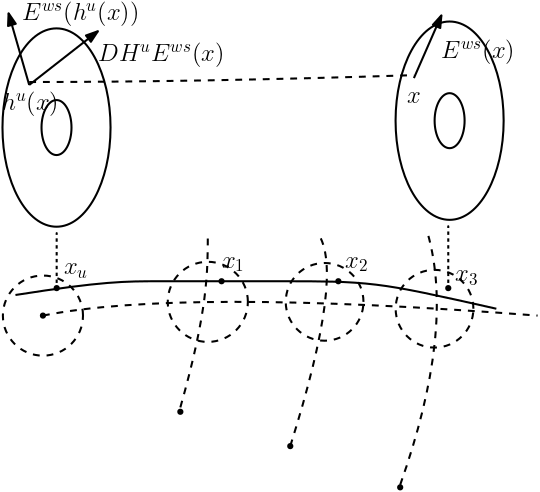}
\caption{Proposition~\ref{prop.mainprop}}
\label{figure-prop}
\end{figure}

Proposition \ref{prop.mainprop} will be proved in Section \ref{section.proofprop}. 

\begin{proof}[Proof of Theorem A assuming Proposition \ref{prop.mainprop}]

By Horita-Sambarino \cite{horitasambarino}, there exists a $C^1$-open and $C^r$-dense subset $\mathcal{U}_1$ of $\mathcal{A}^r$ consisting of diffeomorphisms which are accessible. Let $f\in \mathcal{U}_1$ and let $g \in \mathcal{U}_1$ be a diffeomorphism $C^r$-close to $f$ given by Proposition \ref{prop.mainprop}.  In particular, there are points $q_u, q_1, q_2, q_3\in \tor^2$ in the same unstable manifold for $g_1$ and neighborhoods $V_i$ of $q_i$, for $i=u,1,2,3$, such that conclusion (1) from Proposition \ref{prop.mainprop} holds.

Let $\mu$ be an ergodic $u$-Gibbs measure. Since $g$ is accessible, $\mu$ cannot verify condition (3) in Theorem \ref{thm.urigidityanosov}. 

Fix $x \in \tor^4$ a typical point for $\mu$. Notice that $\pi_1(W^{uu}(x,g)) = W^u(x_1,g_1)$ and that $W^u(x_1,g_1)$ is dense in $\tor^2$ by the minimality of the unstable foliation of an Anosov diffeomorphism in $\tor^2$. 

Since $\mu$ is a $u$-Gibbs measure, we can find points $x_i = (x_1^i, x_2^i)\in W^{uu}(x,g)$ which are $\mu$-typical and such that $x_1^i \in V_i$ for $i=u,1,2,3$. Property (1) of Proposition \ref{prop.mainprop} implies that there exist $j,l \in \{u,1,2,3\}$ such that
\[
DH^u_{x_i,x_j,g}(p_j). E^{ws}_g(x_j) \neq E^{ws}_g(x_l).
\]
By Theorem \ref{thm.urigidityanosov}, $\mu$ is SRB. Since $g$ is a transitive Anosov system, there is only one SRB measure and hence $\mu$ is the unique SRB measure and it has full basin (see \cite{bowenbook}).  

Item (2) in Proposition \ref{prop.mainprop}, implies that this conclusion is $C^2$-open in $\mathcal{A}^r$. Thus, we obtain a set $\mathcal{U} \subset \mathcal{U}_1$ which is $C^2$-open and $C^r$-dense in $\mathcal{A}^r$ for which the conclusion of Theorem A holds.

\end{proof}

\section{Transversality } \label{section.transversality}
 
Let $\mathbb{P}T\tor^2$ be the projectivization of the tangent bundle of $\tor^2$.  Observe that $T\tor^2$ is a trivial bundle and it can be identified with $\tor^2 \times \mathbb{R}^2$. In particular, $\mathbb{P}T\tor^2$ can be identified with $\tor^2 \times S^1 = \tor^3$.  From now on we will use this identification.

Given a $C^{1+\alpha}$ function $V: \tor^2 \to S^1$, there is a one to one correspondence  with a section of the projective bundle $\mathbb{P}T\tor^2$ given by $x\mapsto (x,V(x))$.  We will call such a function $V$ a \emph{line field}.

A $C^r$-diffeomorphism  $f:\tor^2\to \tor^2$, acts as a $C^{r-1}$-diffeomorphism on $\mathbb{P}T\tor^2$. Given a $C^{1+\alpha}$ line field $V$, we denote by $f_*V$ the pushforward of $V$, which is the line field given by $f_*V(x) = [Df(f^{-1}(x))v(f^{-1}(x))]$, where $[w]$ denotes the projective class of a nonzero vector $w$ and $v(f^{-1}(x))$ is any nonzero vector contained in the direction $V(f^{-1}(x))$.  

The main goal of this section is to prove the following proposition:

\begin{proposition}\label{prop.perturbingvf}
Fix $r\geq 3$. For any $C^{1+\alpha}$ line fields  $W,V_1,V_2,V_3:\tor^2 \to S^1$, there are $C^r$-diffeomorphisms of $\tor^2$,  $h_1,h_2,h_3$,  arbitrarily $C^r$ close to the identity such that $$\{W(x)\}\cap \{(h_1)_*V_1(x)\}\cap \{(h_2)_*V_2(x)\} \cap \{(h_3)_*V_3(x)\}=\emptyset, \textrm{ }  \forall x\in \tor^2.$$
\end{proposition}

To prove Proposition \ref{prop.perturbingvf}, we will need the following theorem, which is a consequence of Sard's Lemma.

\begin{theorem}[\cite{hirsch}, Theorem 2.7]\label{th.hirsch.trans}
Let $\cP,M,N$ be $C^r$ manifolds without boundary and $A\subset N$ a $C^r$ submanifold. Let $F:M \times \cP\to N$ be a $C^r$ map that satisfies the following conditions:
\begin{enumerate}
\item $F$ is transverse to $A$;
\item $r>\max\{0,dim M+ dim A-dim N\}$.
\end{enumerate}
Then the set 
$$\{\Theta\in \cP:F_\Theta \pitchfork A\}$$
is open and dense.
\end{theorem}

\begin{proof}[Proof of Proposition \ref{prop.perturbingvf}]

Consider the $C^{1+\alpha}$-submanifold of $\tor^3$ given by the graph of $W$,  $graph(W)$, and let $A=graph(W)\times graph(W)$ the $4$-dimensional $C^{1+\alpha}$ submanifold of $\tor^6$. We are going to construct a $C^{1}$ map $F:\tor^2\times \cP\to \tor^6$, where $\cP$ is a set of parameters, such that $F$ is transverse to $A$.  We will then apply Theorem \ref{th.hirsch.trans} for $r=1$, $M = \tor^2$ and $N = \tor^6$.

Fix $\rho>0$ small. For each $x\in \tor^2$ and $\theta \in (-\pi,\pi)$ consider the smooth map $R_{x,\theta}: \tor^2 \to \tor^2$ that coincides with a rotation of angle $\theta$ and centered at $x$ in $B_{\rho}(x)$,  it is the identity outside $B_{2\rho}(x)$ and in a neighborhood of $\partial B_{2\rho}(x)$.  Observe that $R_{x,\theta}$ varies smoothly with the choice of $\theta$.  Indeed, $R_{x,\theta}$ can be obtained as the time $\theta$ flow of a smooth vector field on $\tor^2$. 

For each $x\in \tor^2$ and $i=1,2$, consider the map $\Phi_x^i(y,\theta) = (R_{x,\theta})_*V_i(y)$.  Since $R_{x,\theta}(y)$ is smooth with the choices of $y$ and $\theta$ and $V_i$ is $C^{1+\alpha}$, we have that $\Phi^i_x$ is $C^{1+\alpha}$, where $D\Phi^i_x$ is $(C',\alpha)$-H\"older for some constant $C'$ that can be chosen independently of $x$. 
Observe that $\Phi^i_x(x,\theta) =  DR_{x,\theta}(x) V_i(x)$, which is a rotation of $V_i(x)$ by an angle $\theta$. In particular,
\[
\displaystyle \frac{\partial \Phi^i_x}{\partial \theta}(x,0) = 1, \textrm{ for $i = 1,2.$}
\]
Since $D\Phi^i_x$ is $(C'.\alpha)$-H\"older, there exists a constant $\tilde{\rho} \in (0,\rho)$ such that for any $y\in B_{\tilde{\rho}}(x)$, we have $\frac{\partial \Phi^i_x}{\partial \theta}(y,0) >\frac{1}{2}$.  Consider a finite set of points $x_1, \cdots, x_k$ with the property that 
\[
\tor^2 = \displaystyle \bigcup_{j=1}^k B_{\tilde{\rho}}(x_j).
\]

%
%

%

Let $\cP=(-\pi,\pi)^k$ and for each $\Theta=(\theta_1,\dots,\theta_k)\in \cP$,  define $\mathbb{F}_\Theta:\tor^2\to \tor^2$ by
$$\mathbb{F}_\Theta(.)= R_{x_1,\theta_1} \circ \cdots \circ R_{x_k, \theta_k} (.) $$
and $F:\tor^2\times \cP^2\to \tor^3 \times \tor^3$ by
$$F(x,\Theta_1,\Theta_2)=(x,(\mathbb{F}_{\Theta_1})_*V_1(x), x, (\mathbb{F}_{\Theta_2})_*V_2(x)).$$
For $(\Theta_1, \Theta_2) \in \cP^2$ and for any $i\in \{1,2\}$, write $\theta_j^i$ the $j$-th coordinate of $\Theta_i$, for $j=1,\cdots, k$. 

First, let us check that $F$ is transverse to $A$ on $\tor^2\times \{0\}^{2k}$.  Recall  that $A$ is the product of two graphs. In particular, for each $(x,y)\in A$, 
\[
T_{(x,y)}A=T_x graph (W)\times T_y graph (W)
\]
and $graph (W)$ is transverse to $\{\pi_{\tor^2}(x)\}\times S^1$ at $x$, where $\pi_{\tor^2}:\tor^2 \times S^1 \to \tor^2$ is the natural projection.  Recall that the image of $F$ belongs to $\tor^3 \times \tor^3$ and let $F^1$  be the natural projection of $F$ into the first torus.  Similarly, define $F^2$ as the projection of $F$ into the second torus.  Hence,  for $F$ to be transverse to $A$ on $\tor^2 \times \{0\}^{2k}$,  it suffices that for any $(x_1, x_2) \in A$, where $\pi_{\tor^2}(x_1) = \pi_{\tor^2}(x_2) = x$ for some point $x\in \tor^2$, there exists $j,l\in \{1,\cdots , k\}$ such that the third component of the vectors $
\frac{\partial F^1}{\partial \theta^1_j}(x,0,0)$ and $ \frac{\partial F^2}{\partial \theta_l^2}(x,0,0)$ are both non zero.

Let us show this for $F^1$, the argument for $F^2$ is the same.  Let $x = \pi_{\tor^2}(x_1)$ and take $j$ such that $x\in B_{\tilde{\rho}}(x_j)$.  By our choice of $\tilde{\rho}$ and $j$, we have that 
\[
\frac{\partial F^1}{\theta^1_j}(x,0,0) = \left(0,0, \frac{\partial \Phi^1_{x_j}}{\partial \theta}(x,0)\right) \neq (0,0,0).
\]
Hence, $F$ is transverse to $A$ on $\tor^2 \times \{0\}^{2k}$.  Since transversality is open, there exists an open subset of $\cP'^2 \subset \cP^2$, containing $\{0\}^{2k}$ with the property that $F$ is transverse to $A$ restricted to $\tor^2 \times \cP'^2$.

%

Observe that $dim \tor^2 + dim A - dim \tor^6 = 2 + 4 - 6 = 0$. In particular $r=1$ is enough to apply Theorem \ref{th.hirsch.trans}.  Hence, there exists $(\Theta_1, \Theta_2)\in \cP'^2$ such that $F(., \Theta_1, \Theta_2):\tor^2\to \tor^6$ is transverse to $A$.  Take $h_1=\mathbb{F}_{\Theta_1}$ and $h_2=\mathbb{F}_{\Theta_2}$, and observe that since $(\Theta_1, \Theta_2)$ can be taken arbitrarily close to $\{0\}^{2k}$, the diffeomorphisms $h_1$ and $h_2$ can be taken arbitrarily $C^r$-close to the identity. 

Let $Z=\{x\in \tor^2:W(x)=(h_1)_*V_1(x)=(h_2)_*V_2(x)\}$, if $x\in Z$ we have that $F(x,\Theta_1,\Theta_2)\in A$,  and by transversality,  $Z$ is finite. Now it is straightforward to take $h_3$ such that $(h_3)_*V_3(x)\neq W(x)$ for every $x\in Z$, which concludes the proof of the proposition.  \qedhere

\end{proof}

\section{Proof of Proposition \ref{prop.mainprop}}  \label{section.proofprop}

Let $f\in \mathcal{A}^r$ and $q_u,p_1, p_2, p_3, q_1, q_2,q_3$ be as in the statement of Proposition \ref{prop.mainprop}, and fix $\varepsilon>0$.  For any $z\in \tor^4$, define $E^{ws}_{z,f} := E^{ws}|_{W^c(z)}$.

Let $W =  E^{ws}_{q_u,f}$ and $V_i = (H^u_{i,u,f})_* E^{ws}_{q_i,f}$ for $i=1,2,3$.   By Proposition \ref{prop.perturbingvf}, we can find diffeomorphisms $h_i:\tor^2 \to \tor^2$ arbitrarily $C^r$-close to the identity such that for every $x\in \tor^2$
\begin{equation}\label{eq.emptyintersection}
\{W(x) \} \cap \{ (h_1)_* V_1(x) \} \cap \{(h_2)_*V_2(x) \} \cap \{(h_3)_* V_3(x)\} = \emptyset.
\end{equation}

Write $\hat{f}_i = f|_{W^c(q_i)}$ and consider $\hat{g}_i:\tor^2 \to \tor^2$ given by
\[
\hat{g}_i = \hat{f}_i \circ H^u_{u, i, f} \circ h_i^{-1} \circ H^u_{i,u,f},
\]
for $i=1,2,3$.  Since $h_i$ can be taken arbitrarily $C^r$ close to the identity and $(H^u_{i,u,f})^{-1} = H^u_{u,i,f}$,  we can suppose that $\hat{g}_i$ is also $C^r$ close to $\hat{f}_i$.  Hence, we can perturb $f$ to obtain a  $C^r$ diffeomorphism $g$ with the following properties:
\begin{itemize}
\item $d_{C^r}(f,g) < \varepsilon$;
\item $g = f$ outside a small neighborhood $U$ of the fibers $W^c(q_i)$, for $i=1,2,3$. We can take $U$ small enough so that $U \cap (W^c(f_1^n(q_i)))_{n\in \mathbb{Z}} = W^c(q_i)$;
\item $g_1 = f_1$, in other words, the perturbation only happens in the fibers;
\item $g|_{W^c(q_i)} = \hat{g}_i$.
\end{itemize}

Observe that $E^{ws}_{g_1(q_i),g}$ only depends on the future iterates of $g|_{W^c(g_1^{n+1}(q_i))}$, for $n\geq 0$, which are the same as $f|_{W^c(f_1^{n+1}(q_i))}$ .  Therefore, 
\begin{equation}\label{eq.samews}
E^{ws}_{g_1(q_i), g} = E^{ws}_{f_1(q_i),f}.
\end{equation}

The unstable holonomy $H^u_{u, i,f}$ only depends on $f|_{W^c(f_1^{-n}(q_i))}$ and \linebreak $f|_{W^c(f_1^{-n}(q_u))}$, for $n\geq 0$, which coincides with $g|_{W^c(g_1^{-n}(q_i))}$ and \linebreak $g|_{W^c(g_1^{-n}(q_u))}$. Hence,
\begin{equation}\label{eq.sameuholonomy}
H^u_{u,i, f} = H^u_{u, i, g}, \textrm{ for $i=1,2,3$.}
\end{equation}

Thus, by \eqref{eq.samews} and \eqref{eq.sameuholonomy},
\[
\begin{array}{rcl}
(H^u_{i, u,g})_*(\hat{g}^{-1}_i)_* E^{ws}_{g_1(q_i),g} & = & (H^u_{i,u,g})_* (H^u_{u,i,f})_* (h_i)_*(H^u_{i,u,f})_* (\hat{f}^{-1}_i)_* E^{ws}_{g_1(q_i),g} \\
& = &(H^u_{i,u,f})_* (H^u_{u,i,f})_* (h_i)_*(H^u_{i,u,f})_* (\hat{f}^{-1}_i)_* E^{ws}_{f_1(q_i),f}\\
 & = & (h_i)_* (H^u_{i,u,f})_* E^{ws}_{q_1,f} \\
 & = & (h_i)_* V_i.
\end{array}
\]

Therefore, $g$ realizes the property \eqref{eq.emptyintersection}. Notice that this property is open, in particular, this implies item $(1)$ of Proposition \ref{prop.mainprop}. Since the direction $E^{ws}$ changes continuously with the diffeormorphism and the unstable holonomies $H^u$ changes continuously in the $C^1$-topology with the choice of $h$ (see Section \ref{subsec.holonomies}), item $(2)$ of Proposition \ref{prop.mainprop} follows.

\begin{remark}
As, we mentioned before, Theorem \ref{thm.urigidityanosov} holds for any $u$-Gibbs measure with one positive and one negative center exponent, where $E^{ws}$ is replaced by the corresponding Oseledets direction. 
The main technical properties that we need in the proof of Theorem~\ref{thm.maintheorem} are the $C^1$ regularity of $E^{ws}$ on the center manifolds over $q_u, p_1, p_2, p_3$ and the dominated splitting.

The same proof of Theorem A, gives the following result. 
\begin{theorem}\label{thm.technical}
Let $f\in \operatorname{Ph}^3$.  If $f_1$ has periodic points $q_u, p_1, p_2, p_3$ such that $E^{ws}$ is $C^1$ restricted to $W^c(p_i)$, for $i=u,1,2,3$.
Then there exists a neighborhood $\mathcal{U}$ of $f$ inside $\sk$,  and a subset $\mathcal{V} \subset \mathcal{U}$ which is $C^2$-open and $C^r$-dense in $\mathcal{U}$ with the following property.  For any $g\in \mathcal{V}$, any ergodic $u$-Gibbs measure for $g$ with one positive and one negative center exponent is SRB.
\end{theorem} 
For surface diffeomorphisms having a dominated splitting, in general, without any uniform contraction assumption on $E^{ws}$, this distribution is only H\"older continuous.
\end{remark}

\section{Minimality of the strong unstable foliation}\label{section.minimality}

Throughout this section,  let $\mathcal{U}$ be the $C^r$-dense and $C^2$-open set obtained in Theorem A and  fix $g\in \mathcal{U}$. The goal of this section is to prove that the strong unstable foliation of $g$ is minimal. 

\begin{lemma}
\label{lem.density}
For every $\varepsilon>0$, there exists $m = m(\varepsilon)$ such that for any $p\in \tor^4$ the set $\displaystyle \bigcup_{j=0}^m g^j(W^{uu}_{1}(p,g))$ is $\varepsilon$-dense in $\tor^4$.
\end{lemma}
\begin{proof}
Suppose not, then there exist $\varepsilon>0$, sequence of points $(p_n)_{n\in \mathbb{N}}$, $(q_n )_{n\in \mathbb{N}}$ and a sequence of natural numbers $(k_n)_{n\in \mathbb{N}}$,  with $k_n \to +\infty$, such that
\[
\displaystyle \left( \bigcup_{j=0}^{k_n} g^j(W^{uu}_1(p_n,g)) \right) \cap B(q_n, \varepsilon) = \emptyset. 
\]
Without loss of generality, we can suppose that  there are points $p$ and $q$ such that $\lim_{n\to +\infty} p_n = p$ and $\lim_{n\to +\infty} q_n = q$.  It is easy to see that 
\begin{equation}\label{eq.nointersection}
\displaystyle \left( \bigcup_{j=0}^{+\infty} g^j(W^{uu}_1(p,g)) \right)  \cap B(q, \varepsilon/2) = \emptyset.
\end{equation}
Let $\mu$ be the unique $u$-Gibbs measure of $g$, by Theorem A, and since $\mathrm{supp}(\mu) = \tor^4$ we conclude that $\mu(B(q,\varepsilon/2))>0$. For each $n \in \mathbb{N}$, consider
\[
\mu_n:= \displaystyle \frac{1}{n} \sum_{j=0}^{n-1} g^j_*(m^u_p),
\]
where $m^u_p$ is the normalized Lebesgue measure on $W^{uu}_1(p,g)$.  By Theorem A, $\lim_{n\to +\infty} \mu_n = \mu$. Since $B(q, \varepsilon/2)$ is an open set, we obtain that 
$\liminf_{n\to +\infty} \mu_n(B(q,\varepsilon/2)) \geq \mu(B(q,\varepsilon/2)) >0.$ In particular, for $n$ large enough,
$\mu_n(B(q,\varepsilon/2))>0$. This is a contradiction with \eqref{eq.nointersection}. \qedhere
\end{proof}

Recall that the set $\mathcal{U}$ of Theorem A was obtained by  applying Proposition \ref{prop.mainprop}. In particular, we can assume that there are points $q_u, q_1, q_2,q_3 \in \tor^2$ that verify the conclusion of Proposition \ref{prop.mainprop}.  For every $q \in W^u(q_u,f_1)$ and  $x\in \tor^2$, consider the angle function 
\[
\alpha(x,q) := \angle (E^{ws}_{q_u}(x), DH^{u}_{q,q_u}(H^u_{q_u,q}(x))E^{ws}_q(H^u_{q_u,q}(x))).
\]
Define 
\[
\alpha(x):= \max\{\alpha(x,q_1), \alpha(x,q_2), \alpha(x,q_3)\}. 
\]
Since the unstable holonomy $H^u$ between center manifold is $C^1$ and the direction $E^{ws}$ is continuous, the function $\alpha(\cdot)$ is continuous. Moreover, by the conclusion of Proposition \ref{prop.mainprop}, for every $x\in \tor^2$, $\alpha(x)>0$. By compactness, there exists a constant $\alpha>0$ such that $\alpha(x)\geq \alpha$ for every $x\in \tor^2$.

Fix an orientation of $W^u_R(q_u,g_1)$.  Without loss of generality, we may suppose that $q_u <q_1<q_2<q_3$, we have the freedom to choose the points $q_i$ before performing the perturbation of Proposition \ref{prop.mainprop}.  
Let $R =d^u(q_u,q_3)$, this is the unstable distance between $q_u$ and $q_i$.
For each $x\in \tor^2$, let $q_x \in W^u_R(q_u,g_1)$ be the smallest $q>q_u$ such that $\alpha(x,q_x) = \alpha$.  Consider the set 
\[
\mathcal{D}_x := \{H^u_{q,q_u}(W^{ws}_{r_0}(H^u_{q_u,q}(x))): q_u \leq q \leq q_x\},
\] 
for some $r_0>0$.

\begin{lemma}
\label{lem.uniforminterior}
There exist  constants $\varepsilon_c>0$ and $r_0>0$ with the following property: for any $x\in \tor^2$ the set $\mathcal{D}_x$ contains a ball of radius $\varepsilon_c$, $B(y_x,\varepsilon_c)$, for some point $y_x\in \mathcal{D}_x$.  Moreover,  we can take $B(y_x,\varepsilon_c)$ such that the distance between the boundary of this ball and the boundary of $\mathcal{D}_x$ is greater than $2\varepsilon_c$. 
\end{lemma}
\begin{proof}

Fix $x\in \tor^2$, using the fact that the map $H^u_{q_u,q}$ changes continuously in the $C^1$-topology with $q$, we can fix $r_0>0$ small enough such that for each $q\in W^u_R(q_u,g_1)$,  the set $H^u_{q,q_u}(W^{ws}_{r_0}(H^u_{q_u,q}(x)))$ is a curve of length greater than $Cr_0$, for some constant $C>0$, and it is tangent to a small cone containing $DH^u_{q,q_u}(H^u_{q_u,q}(x)) E^{ws}_q(H^u_{q_u,q}(x))$.  Moreover, the map $W^u_R(q_u,g_1) \ni q \mapsto H^u_{q,q_u}(W^{ws}_{r_0}(H^u_{q_u,q}(x)))$ is continuous.  Since $\alpha(x,q_x) = \alpha$, we have that the Hausdorff distance between $W^{ws}_{r_0}(x)$ and $ H^u_{q_x,q_u}(W^{ws}_{r_0}(H^u_{q_u,q_x}(x)))$ is uniformly bounded from below by a constant depending on $\alpha$.  Hence,  there exists $\varepsilon_c>0$ such that $\mathcal{D}_x$ contains a ball of radius $\varepsilon_c$ centered at some point $y_x\in \mathcal{D}_x$ (see Figure \ref{figure1}).  Furthermore, up to replacing $\varepsilon_c$ by $\varepsilon_c/4$, we can assume that the boundary of $B(y_x, \varepsilon_c)$ is  at least $2\varepsilon_c$ distant from  the boundary of $\mathcal{D}_x$.  Moreover,  by compactness, we can choose $\varepsilon_c$ to be uniform independent of $x$.  \qedhere

\end{proof}

\begin{figure}[h]
\centering
\includegraphics[width= 0.8\textwidth]{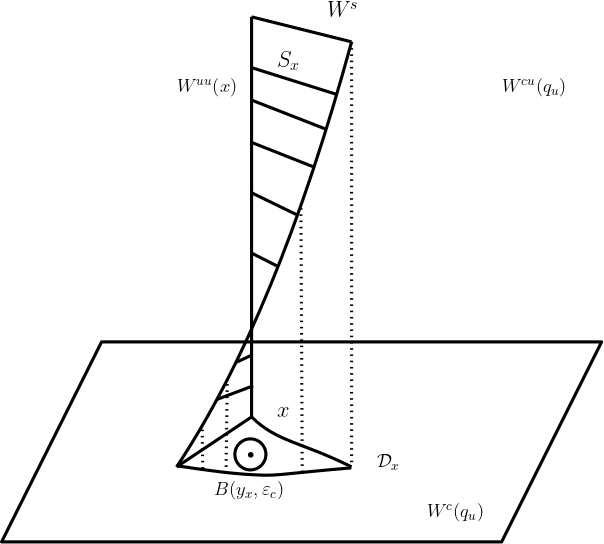}
\caption{Configuration of Lemma \ref{lem.uniforminterior}}
\label{figure1}
\end{figure}

For each $x\in \tor^2$, define the map $\mathcal{R}_x: W^u_R(q_u,g_1) \times \tor^2 \times (-\pi/2 , \pi/2) \to \tor^2$ as follows. For each $q \in W^u_R(q_u,g_1)$ and $\theta \in (-\pi/2, \pi/2)$, the map $\mathcal{R}_x(q,.,\theta)$ coincides with the rotation of angle $\theta$ centered at the point $H^u_{q_u,q}(x)$ inside the ball $B(H^u_{q_u,q}(x),2r_0) \subset \tor^2$, and it coincides with the identity outside $B(H^u_{q_u,q}(x),4r_0)$.  The map $\mathcal{R}_x$ can be obtained as the flow generated by a vector field that is tangent to the center manifolds (the vertical tori). Moreover, since the strong unstable manifold and the center manifolds are $C^1$, we can take $\mathcal{R}_x$ to be $C^1$.

Define $S_x:= \{ W^{ws}_{r_0}(H^u_{q_u,q}(x)): q_u\leq q \leq q_x\}$ and observe that $S_x$ is a topological surface.   We can assume that $\alpha>0$ is small enough such that for any $q\in [q_u,q_x]$ and $y\in H^u_{q_u,q}(\mathcal{D}_x)$, there exists a unique $\theta(q,y) \in (-\pi/2, \pi/2)$ such that $\mathcal{R}_x(q, y, \theta(q,y)) \in S_x$. Observe that if $\theta(q,y)=0$ this implies that $(q,y)\in S_x$.

 For any $\varepsilon \in (0, \varepsilon_c)$, consider the set
\[
\mathcal{B}^u(y_x, \varepsilon):= \displaystyle \bigcup_{y\in B(y_x,\varepsilon)} W^{uu}_{[q_u,q_x]} (y,g), 
\]
where $W^{uu}_{[q_u,q_x]}(y,g)$ is the segment of strong unstable manifold that projects into the piece of unstable manifold of $W^u(q_u,g_1)$ containing $q_u$ and $q_x$ as the boundary points.  Define the sets $D^-_x(\varepsilon):= B(y_x,\varepsilon)$ and $D^+_x:= H^u_{q_u, q_x}(B(y_x,\varepsilon))$. 

\begin{lemma}
\label{lem.tvi}
Let $\gamma:[0,1] \to \mathcal{B}^u(y_x, \varepsilon_c)$ be any continuous curve such that $\gamma(0) \in D_x^-$ and $\gamma(1) \in D^+_x$. Then there exists $t_0\in [0,1]$ such that $\gamma(t_0) \in S_x$. 
\end{lemma}
\begin{proof}
Without loss of generality, we may suppose that for any point $p\in \mathcal{D}_x$ we have that $\theta(q_u,p) \geq 0$ and $\theta(q_x, H^u_{q_u,q_x}(p)) \leq 0$. 

 Since the distance between the boundary of $B(y_x,\varepsilon_c)$ and the boundary of $\mathcal{D}_x$ is greater than $2\varepsilon_c$, there exists a constant $\rho>0$ such that for any point $y\in B(y_x, \varepsilon_c)$ and $z\in H^u_{q_u,q_x}(B(y_x,\varepsilon_c))$ we have  $\theta(q_u, y) > \rho$ and $\theta(q_x, z) < -\rho$.  This constant can be taken independently of $x$. 
 
Observe that $\gamma(t)\in W^u_R(q_u,g_1)\times H^u_{q_u,q(t)}(\mathcal{D}_x)$, where $q(t)$ is the projection of $\gamma(t)$ into the base torus, then we can consider the function $\theta(t) := \theta(\gamma(t))$. Since $\gamma$ is continuous, and $\mathcal{R}_x$ is $C^1$, the function  $\theta(t)$ is a continuous in $t$. In particular, $\theta(0)>\rho$ and $\theta(1) <-\rho$. Hence, there exists $t_0$ such that $\theta(t_0) = 0$, which implies that $\gamma(t_0) \in S_x$. \qedhere

\end{proof}

\begin{figure}[h]
\centering
\includegraphics[width= 0.8\textwidth]{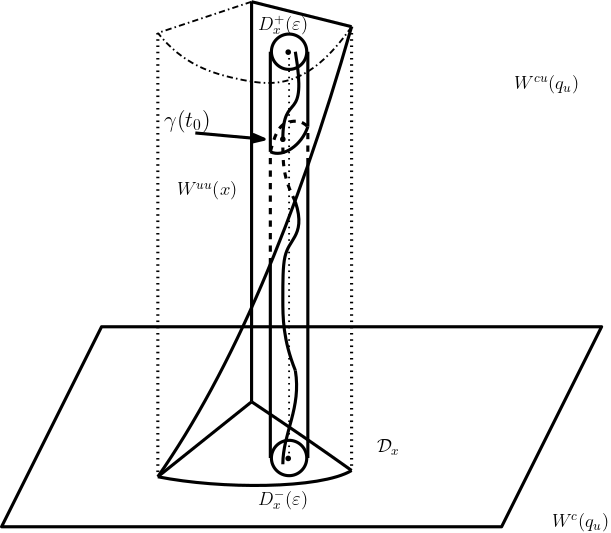}
\caption{Lemma \ref{lem.tvi}}
\label{figure2}
\end{figure}

For each $p\in \tor^2$, write $W^{cu}_R(p):= \bigcup_{x\in \tor^2} W^{uu}_{R}(x,g)$.  So far, we have been using the stable holonomies between two center manifolds. In this section, we will use the stable holonomy between two center unstable manifolds.  For  $p_s\in W^s_{loc}(q_u, g_1)$, let $H^s_{p_s, q_u}$ be the stable holonomy between $W^{cu}_R(p_s)$ and $W^{cu}_R(q_u)$. It is well known that $H^s_{p_s,q_u}$ is continuous, indeed, it is H\"older continuous. 

 For each $y\in \tor^2$, $\varepsilon>0$ and $\delta>0$, define the set 
 \[
 \mathcal{D}^s(y, \varepsilon, \delta):= \displaystyle \bigcup_{z\in B(y,\varepsilon)} W^{ss}_{\delta}(z,g).
 \]
Since the strong stable foliation is $C^1$ inside a center stable manifold, the set $\mathcal{D}^s(y,\varepsilon, \delta)$ is a $C^1$ submanifold. 

\begin{lemma}
\label{lem.stableintersection}
There exists $\delta_s>0$ such that for any $x\in \tor^2$ and any $z\in \mathcal{D}^s(y_x, \varepsilon_c/4, \delta_s)$, it holds that
\[
H^s_{z_s, q_u}(W^{uu}_{R}(z,g)) \cap S_x \neq \emptyset,
\]
where $z_s$ is the projection of $z$ into the base torus. 
\end{lemma}

\begin{proof}
Observe that if $z$ is close to $y_x$ then $W^{uu}_{R}(z,g)$ is uniformly $C^1$-close to $W^{uu}_{R}(y_x,g)$. Moreover, the map $H^s_{z_s,q_u}$ is uniformly $C^0$-close to the identity.  Hence, if $\delta_s>0$ is sufficiently small, then $H^s_{z_s, q_u}(W^{uu}_{R}(z,g))$ is a continuous curve  $C^0$-close to $W^{uu}_{R}(y_x,g)$. In particular, we can find a continuous curve $\gamma \subset H^s_{z_s, q_u}(W^{uu}_{R}(z,g))$ which is contained in $\mathcal{B}^u(y_x,\varepsilon_c)$ and intersects $D^-_x(\varepsilon_c)$ and $D^+_x(\varepsilon_c)$. The proof follows from  Lemma \ref{lem.tvi}. \qedhere
 
\end{proof}

\begin{proof}[Proof of Theorem B]

A standard application of Zorn's lemma implies that we can always find a compact set $L \subset \tor^4$ which is $u$-minimal, that is, $L = \overline{W^{uu}(p,g)}$ for every $p\in L$.  Let us fix a $u$-minimal set $L$. We will split the proof into two cases.

\subsubsection*{Case 1:} $L$ is periodic. 

In this case, there exists $k\in \mathbb{N}$ such that $g^k(L) = L$.  Suppose that $k$ is the smallest number with this property.  The set $\mathcal{L} = L \cup \cdots \cup g^{k-1}(L)$ is a compact, $u$-saturated and $g$ invariant set. Hence, there exists a $u$-Gibbs measure $\mu$ supported on $\mathcal{L}$. By Theorem A,  and since the support of the SRB measure of $g$ is $\tor^4$, we conclude that $\mathcal{L} = \tor^4$.  

If $k=0$ then $L = \tor^4$ and we are done. Otherwise, we obtain that $\tor^4$ is decomposed into a finite union of compact subsets. Hence,  there exists $i<k-1$ such that $L \cap g^i(L) \neq \emptyset$. The set $g^i(L)$ is also $u$-minimal and two $u$-minimal sets are either the same  or disjoint.  Therefore, $g^i(L) = L$ but this contradicts our choice of $k$. This finishes the proof of the periodic case. 

\subsubsection*{Case 2:} $L$ is aperiodic. 

In this case the sequence of sets $\left(g^n(L)\right)_{n\in \mathbb{Z}}$ is pairwise disjoint. 

\begin{claim}
\label{claim.claim1}
There exist a number $k\in \mathbb{N}$ and a sequence $(n_j)_{j\in \mathbb{N}}$ such that $n_j \to +\infty$, as $j\to +\infty$, and $W^s_{loc}(g^{-n_j}(L),g) \cap g^{-n_j + k}(L) \neq \emptyset$. 
\end{claim}

\begin{proof}
Let $\varepsilon_c$ and $\delta_s$  be as in Lemmas \ref{lem.uniforminterior} and \ref{lem.stableintersection}.  Since the unstable foliation of an Anosov diffeomorphism in $\tor^2$ is minimal, for every $n\in \mathbb{N}$ the projection of $g^{-n}(L)$ into the base torus is $\tor^2$. In particular, for every $n\in \mathbb{N}$ we can choose a point $x_n \in W^c(q_u) \cap g^{-n}(L)$.  Let $y_n\in W^c(q_u)$ be the point given by Lemma \ref{lem.uniforminterior} for the point $x_n$ and consider $\mathcal{D}^s(y_n, \varepsilon_c/4, \delta_s)$.  Observe that for every $y\in W^c(q_u)$, the set $\mathcal{D}^s(y, \varepsilon_c/4, \delta_s)$ is a $C^1$ submanifold of uniform size and uniformly transverse to the strong unstable foliation. 

Fix $\varepsilon>0$ small and by Lemma \ref{lem.density}, there exists $m= m(\varepsilon)$ such that for every $n\in \mathbb{N}$ the set
\[
\displaystyle \bigcup_{j=0}^{m-1} g^{-n + j}(L)
\]
is $\varepsilon$-dense.  Hence, there exists $k_n\in \{0, \cdots, m-1\}$ such that $g^{-n + k_n}(L) \cap \mathcal{D}^s(y_n, \varepsilon_c/4, \delta_s) \neq \emptyset$. By Lemma \ref{lem.stableintersection},  $g^{-n + k_n}(L) \cap W^s_{loc}(g^{-n}(L), g) \neq \emptyset$.
 
 By the pigeonhole principle, we can find $k\in \{0, \cdots, m-1\}$ and a sequence $n_j \to +\infty$ that verify the conclusion of the claim. \qedhere
\end{proof}

By Claim \ref{claim.claim1},   there are points $p_{n_j} \in g^{-n_j + k}(L)$  and $p_{n_j}' \in g^{-n_j}(L)$ such that $p_{n_j} \in W^s_{loc}(p_{n_j}', g)$.  Hence, $\lim_{j\to +\infty} d(g^{n_j}(p_{n_j}), g^{n_j}(p_{n_j}') = 0$. Therefore, $g^k(L) \cap L \neq \emptyset$. Since, $L$ is $u$-minimal, we conclude that $g^k(L) = L$ and $L$ is periodic, which is a contradiction. \qedhere

\end{proof}

%
%
%

\information

\end{document}